\documentclass[reqno]{amsart}
\usepackage[top=2cm,bottom=2cm,right=2.5cm,left=2.5cm]{geometry}
\usepackage{amssymb}
\usepackage{amsmath, amsthm, amscd, amsfonts, amssymb, graphicx, color}
\usepackage[bookmarksnumbered, colorlinks, plainpages]{hyperref}
\hypersetup{colorlinks=true,linkcolor=red, anchorcolor=green, citecolor=cyan, urlcolor=red, filecolor=magenta, pdftoolbar=true}
 
\usepackage{hyperref}

\newtheorem{theorem}{Theorem}[section]
\newtheorem{lemma}[theorem]{Lemma}

\newtheorem{corollary}[theorem]{Corollary}
\theoremstyle{definition}
\newtheorem{definition}[theorem]{Definition}
\newtheorem{example}[theorem]{Example}

\theoremstyle{remark}

\numberwithin{equation}{section}

\begin{document}

\setcounter{page}{1}

\title[$\theta-\phi-$contraction on $\left(\alpha,\eta \right)-$complete rectangular $b-$
 metric spaces]{$\theta-\phi-$contraction on $\left(\alpha,\eta \right)-$complete rectangular $b-$%
 metric spaces}

\author[Abdelkarim Kari, Mohamed Rossafi, El Miloudi Marhrani and Mohamed Aamri]{Abdelkarim Kari$^2$, Mohamed Rossafi$^{1*}$,  El Miloudi Marhrani$^2$ {and} Mohamed Aamri$^2$  }

\address{$^{1}$Department of Mathematics, University of Ibn Tofail, B.P. 133, Kenitra, Morocco}
\email{\textcolor[rgb]{0.00,0.00,0.84}{rossafimohamed@gmail.com}}
\address{$^{2}$Laboratory of Algebra, Analysis and Applications Faculty of Sciences Ben M’Sik, HassanII University, Casablanca, Morocco}
\email{\textcolor[rgb]{0.00,0.00,0.84}{abdkrimkariprofes@gmail.com}}
\email{\textcolor[rgb]{0.00,0.00,0.84}{fsb.marhrani@gmail.com}}
\email{\textcolor[rgb]{0.00,0.00,0.84}{aamrimohamed82@gmail.com}}


\date{
\newline \indent $^{*}$Corresponding author}

\begin{abstract}
In this paper, we present some fixed point results for generalized $\theta -\phi -$contraction in the framework of $\left( \alpha ,\eta \right)-$compete rectangular $b-$metric spaces. Further, we establish some fixed point theorems for this type of mappings defined on such spaces. The presented extend, generalize and improve existing results in the literature. Moreover, to support our main results we give an illustrative example. 
\end{abstract} \maketitle

\section{Introduction}
The well known Banach contraction theory is one of the methods used, which states that if $ (X,d) $ is a complete metric space and  $\ T:X\rightarrow X$ is self-mapping with contraction, then $ T $ has a unique fixed point \cite{BAN}. 
 \\In 2000, Branciari \cite{BRA} introduced the notion of generalized metric spaces, such as the triangle inequality is replaced by the inequality $d(x, y) \leq d(x,u) + d(u, v) + d(v, y)$ for all pairwise distinct points $x$, $y$, $u$, $v \in X$. Since then, several results have been proposed by many mathematicians on such spaces( see  \cite{BAR,Li,DAS,DSA,JS,KS}.\\
 The concept of metric space, as an ambient space in fixed point theory, has been generalized in several directions. Such as, $b- $ metric spaces \cite{CZ} and generalized metric spaces.\\
Combining conditions used for definitions of b-metric and generalized metric spaces. Roshan et al \cite{RO} announced the notion of rectangular b-metric space.\\
Hussain and et al \cite{HU} introduced the concept of $ \alpha-\eta- $complete rectangular $ b- $metric space and proved certain results of fixed point theory on such spaces.\\
In this paper, we provide some fixed point results for generalized $\theta -\phi -$contraction in the framework of $\left( \alpha ,\eta \right)-$compete rectangular $b-$metric spaces, also we give two examples to support our results.
  
\section{Preliminaries}
\begin{definition}
\cite{RO}. Let $X$ be a nonempty set, $s\geq 1$ be a given real number, and let
d: $X\times X\rightarrow \left[ 0,+\infty \right[ $
be a mapping such that for all $x,y$ $\in X$ and all distinct points $u,v\in
X,$ each  distinct from $x$ and $y$:\\
\item[1.] $d\left(x, y\right) =0,$ if only if $x=y;$\\
\item[2.] $d\left(x, y\right) =d\left(y, x\right);$\\
\item[3.] $d\left(x, y\right) \leq $ $s\left[d\left( x,u\right)+d\left(u, v\right) +d\left(v, y\right) \right] $ $\left( b-rectan gular\ inequality \right) .$\\

Then $\left(X, d\right) $ is called a $b-$rectangular metric space.
\end{definition}
\begin{lemma}
\cite{RO}. Let $\left(X, d\right) $ be a rectangular b-metric space.\\ 
\begin{itemize}
\item[(a)]  Suppose that sequences $\left( x_{n}\right) $ and $\left(y_{n}\right) $ in $X$ are such that $x_{n}\rightarrow x$ and $y_{n}\rightarrow y$ as $n\rightarrow \infty ,$ with $x\neq y,$ $x_{n}\neq x$ and $y_{n}\neq y$ for all $n\in \mathbb{N}.$ Then we have
\begin{equation*}
\frac{1}{s}d\left( x,y\right) \leq \lim_{n\rightarrow \infty }\inf d\left(x_{n},y_{n}\right) \leq \lim_{n\rightarrow \infty }\sup d\left(x_{n},y_{n}\right) \leq sd\left( x,y\right).\\
\end{equation*}
\item[(b)]  if $y\in X$ and $\left( x_{n}\right) $ is a Cauchy sequence in $X$ with $x_{n}\neq x_{m}$ for any $m,n\in \mathbb{N},$ $m\neq n,$converging to $x\neq y,$ then
\begin{equation*}
\frac{1}{s}d\left( x,y\right) \leq \lim_{n\rightarrow \infty }\inf d\left(x_{n},y\right) \leq \lim_{n\rightarrow \infty }\sup d\left( x_{n},y\right)\leq sd\left( x,y\right),
\end{equation*}
for all $x\in X.$\\
\end{itemize}
\end{lemma}
Zheng et al \cite{ZH} introduced a new type of contractions called $\theta -\phi $-contractions in metric spaces and proved a new fixed point theorems for such mapping.\\
\begin{definition}
\cite{JA} We denote by $\Theta _{\text{ }}$ the set of functions $\theta :\left] 0,\infty \right[ \rightarrow \left] 1,\infty \right[,$\\
satisfying the following conditions:\\
\item[$\left( \theta _{1}\right) $] $\theta$ is increasing ;\\
\item[$\left( \theta _{2}\right) $]  for each sequence $\left( x_{n}\right) $ $\in \left] 0,\infty \right[ ,\lim_{n\rightarrow \infty }\theta \left(x_{n}\right) =1\Leftrightarrow $ $\lim_{n\rightarrow \infty }x_{n} =0;$\\
\item[$\left( \theta _{3}\right) $] $\theta$ is continuous on $\left] 0,\infty \right[$.\\
\end{definition}
\begin{definition}
\cite{ZH} We denote by $\Phi $ the set of functions $\phi $ : $%
\left[ 1,\infty \right[ $ $\rightarrow $ $\left[ 1,\infty \right[ $\\
satisfying the following conditions:\\
\item[$\left( \Phi _{1}\right) $] $\phi $ : $\left[ 1,\infty \right[ $ $\rightarrow \left[ 1,\infty \right[ $ is non-decreasing;\\
\item[$\left( \Phi _{2}\right) $] for each $\ t>1$, $lim_{n\rightarrow \infty}\phi ^{n}(t)=1$;\\
\item[$\left( \Phi _{3}\right) $] $\phi $ is continuous on $\left[ 1,\infty \right[ $.\\
\end{definition}
\begin{lemma}
\cite{ZH}. If $\phi \in \Phi $, then $\phi (1)=1$ and $%
\varphi (t)<t$ for each $t>1$.\\
\end{lemma}
In 2014 Hussain et al. \cite{H} introduced a weaker notion than the concept of completeness and called it $ \alpha $-completeness for a metric spaces. 
	\begin{definition}
\cite{H}. Let $T:X\rightarrow X$ and $\alpha,\eta $ :$X\times X\rightarrow \left[ 0,+\infty \right[.$ We say that $T$
is a triangular $\left( \alpha ,\eta \right) -$admissible mapping if
\end{definition}
$\left( T_{1}\right) $ $\alpha \left( x,y\right) \geq 1$ $\Rightarrow $ $\alpha \left( Tx,Ty\right) \geq 1,$ $x,y\in X$;\\
$\left( T_{2}\right) $ $\eta \left( x,y\right) \leq 1$ $\Rightarrow $ $\eta\left( Tx,Ty\right) \leq 1,$ $x,y\in X$;\\
\bigskip $\left( T_{3}\right) \left\{ \begin{array}{c}\alpha \left( x,y\right) \geq 1 \\ 
\alpha \left( y,z\right) \geq 1\end{array}\right. \Rightarrow \alpha \left( x,z\right) \geq 1$ for all $x,y,z\in X$;\\
$\left( T_{4}\right) \left\{ \begin{array}{c}\eta \left( x,y\right) \leq 1 \\ \eta \left( y,z\right) \leq 1\end{array}\right. \Rightarrow \eta \left( x,z\right) \leq 1$ for all $x,y,z\in X$.\\
\begin{definition}
\cite{H}. Let $\left( X,d\right) $ be a b-rectangular metric space and let $\alpha ,\eta $ :$X\times X\rightarrow \left[ 0,+\infty \right[
$ be two mappings. The space is said to be:\\ 
\item[(a)] T is $\alpha -$continuous mapping on $\left( X,d\right) ,$ if for given
point $x\in X$ and sequence $\left( x_{n}\right) $ in $X,$
$x_{n}\rightarrow x$ and $\alpha \left( x_{n},x_{n+1}\right) \geq 1$ for all $n\in \mathbb{N},$ imply that $Tx_{n}\rightarrow Tx$ $.$\\
\item[(b)] T is $\eta $ sub$-$continuous mapping on $\left( X,d\right) ,$ if for
given point $x\in X$ and\\ sequence $\left( x_{n}\right) $ in $X,$ $x_{n}\rightarrow x$ and $\eta \left( x_{n},x_{n+1}\right) \leq 1$ for all $ n\in \mathbb{N},$ imply that T$x_{n}\rightarrow Tx$$.$\\
\item[(c)] T is $\left( \alpha ,\eta \right) $ $-$continuous mapping on $\left(X,d\right) ,$ if for given point $x\in X$ and\\ sequence $\left( x_{n}\right) $ in $X,$ $x_{n}\rightarrow x$ and $\alpha \left( x_{n},x_{n+1}\right) \geq 1$ or $ \eta \left( x_{n},x_{n+1}\right) \leq 1$ for all $n\in \mathbb{N},$ imply that $Tx_{n}\rightarrow Tx$$.$
\end{definition}
The following definitions was given by Hussain et al \cite{HU}.
\begin{definition}
\cite{HU}. Let $d\left( X,d\right) $ be a rectangular b-metric space and let
 $\alpha ,\eta $ :$X\times X\rightarrow \left[ 0,+\infty \right[ $ be two mappings. The space $ X $ is said to be:\\
\item[(a)] $\alpha -$complete$,$ if every Cauchy sequence $\left( x_{n}\right) $ in $X$ with $\alpha \left( x_{n},x_{n+1}\right) \geq 1$ for all $n\in \mathbb{N},$ converges in $X.$\\
\item[(b)] $\eta -\sup -$complete$,$ if every Cauchy sequence $\left( x_{n}\right) $ in $X$ with $\eta \left( x_{n},x_{n+1}\right) \leq 1$ for all $n\in  \mathbb{N},$ converges in $X.$\\
\item[(c)] $\left( \alpha ,\eta \right)-$complete$,$ if every Cauchy sequence $ \left( x_{n}\right) $ in $X$ with $\alpha \left( x_{n},x_{n+1}\right) \geq 1 $ or $\eta \left( x_{n},x_{n+1}\right) \leq 1$ for all $n\in \mathbb{N},$ converges in $X.$\\
\end{definition}
\begin{definition}
\cite{HU}. Let $\left( X,d\right) $ be a rectangular b-metric space and let $\alpha ,\eta 
$ :$X\times X\rightarrow \left[ 0,+\infty \right[ $ be two mappings. The
space $ X $ is said to be:\\
\item[(a)] $\left( X,d\right) $ is $\alpha $ -regular, if $x_{n}\rightarrow x$,
where $\alpha \left( x_{n},x_{n+1}\right) \geq 1$ for all $n\in \mathbb{N},$ implies $\alpha \left( x_{n},x\right) \geq 1$ for all $n\in \mathbb{N}.$\\
\item[(b)]$\left( X,d\right) $ is $\eta -$sub -regular, if $x_{n}\rightarrow x$,
where $\eta \left( x_{n},x_{n+1}\right) \leq 1$ for all $n\in \mathbb{N},$ implies $\eta \left( x_{n},x\right) \leq 1$ for all $n\in \mathbb{N}.$\\
\item[(b)] $\left( X,d\right) $ is $\left( \alpha ,\eta \right) $-regular, if $x_{n}\rightarrow x$,
where  $\alpha \left( x_{n},x_{n+1}\right) \geq 1$ or $ 
\eta \left( x_{n},x_{n+1}\right) \leq 1$ for all $n\in \mathbb{N},$ imply that $\alpha \left( x_{n},x\right) \geq 1$ or $\eta \left( x_{n},x\right) \leq 1$ for all $n\in \mathbb{N}.$\\ 
\end{definition}
\section{main results}
\begin{definition}
 Let $d\left( X,d\right) $ be a $ (\alpha, \eta) $- rectangular b-metric space with parameter $s>1$ and let $T$ be a self mapping on $\ X$. Suppose that $\alpha ,\eta
:X\times X\rightarrow \left[ 0,+\infty \right[ $ are two functions. We say that $T$ is an $\left( \alpha ,\eta \right) -\theta -\phi -$contraction, if for all $x,y\in X$ with $\left( \alpha \left( x,y\right) \geq 1\text{ or }\eta \left( x,y\right) \leq 1\right) $ and $d\left( Tx,Ty\right) >0$ we have 
\begin{equation}
\theta \left( s^{2}.d\left( Tx,Ty\right) \right) \leq \phi\left[ \theta \left( \beta_{1}d\left( x,y\right) +\beta _{2}d\left( Tx,x\right) +\beta _{3}d\left(Ty,y\right) +\beta _{4}d\left( y,Tx\right) \right)\right] 
\end{equation}
where $\theta \in \Theta ,\ \phi \in \Phi$, $\beta _{i}\geq0$ for $i\in\{1,2,3,4\},$ $\sum\limits_{\substack{i=0 }}^{i=4} {\beta_i}\leq1$  and  $\beta _{3}<\frac{1}{s}.$
\end{definition}
\begin{definition}
Let $(X,d)$ be a $ (\alpha, \eta) $-complete rectangular $b-$metric space and $T:X\rightarrow X$ \ be a mapping.\\
\item[(1)]$T$ is said to be a $\theta -\phi -$
Kannan- type contraction if there exists $\theta \in \Theta $ and $\phi \in \Phi $ with $\alpha \left( x,y\right) \geq 1$ or $\eta \left(x,y\right) \leq 1$ for any $x,y\in X,$ $Tx\neq Ty$, we have\\ 
\begin{equation*}
\theta \left[ s^{2}d\left( Tx,Ty\right) \right] \leq \phi \left[ \theta\left(\frac{\left( d\left( Tx,x\right) +d\left( Ty,y\right) \right) }{2s}\right) \right].\\
\end{equation*}
\item[(2)]$T$ is said to be a $\theta -\phi -$ Reich-type contraction \ if there exists $\theta \in \Theta $ and $\phi \in \Phi $ with $\alpha \left( x,y\right) \geq 1$ or $\eta \left( x,y\right) \leq 1$ for any $x,y\in X,$ $Tx\neq Ty$, we have\\
\begin{equation*}
\theta \left[s^{2} d\left( Tx,Ty\right) \right]\leq \phi \left[ \theta \left( \frac{d\left( x,y\right) +d\left( Tx,x\right)+d\left( Ty,y\right) }{3s}\right) \right].\\
\end{equation*}
\item[(3)]$T$ is said to be a Kannan type mapping, that is, if there exists $\alpha \in \left] 0,\frac{1}{2s}\right[ $ with $\alpha \left(x,y\right) \geq 1$ or $\eta \left( x,y\right) \leq 1$ for any $x,y\in X,$ $Tx\neq Ty$, we have\\
\begin{equation*}
s^{2}.d\left( Tx,Ty\right) \leq \alpha \left( d\left( Tx,x\right) +d\left(y,Ty\right) \right).\\
\end{equation*}
\item[(4)]$T$ is said to be a Reich type mapping, that is, if there
exists $\lambda \in \left] 0,\frac{1}{3s}\right[ $ with $\alpha \left(x,y\right) \geq 1$ or $\eta \left( x,y\right) \leq 1$ for any $x,y\in X,$ $Tx\neq Ty$, we have
\begin{equation*}
 s^{2}.d\left( Tx,Ty\right)\leq \lambda\left[  d\left( x,y\right) +d\left( Tx,x\right)+d\left( Ty,y\right)\right]  .\\
\end{equation*}
\end{definition}
\begin{theorem}
Let $\left( X,d\right) $ be a $\left( \alpha ,\eta \right) -$complete rectangular b-metric and let\\ $\alpha ,\eta :X\times X\rightarrow \left[0,+\infty \right[ $ be two functions. Let $T:X\times X\rightarrow X$ be a self mapping satisfying following
conditions:\\
\item[(i)] $T$ is a triangular $\left( \alpha ,\eta \right) -$admissible mapping;\\
\item[(ii)] $T$ is an $\left( \alpha ,\eta \right)-\theta -\phi -$contraction;\\
\item[(iii)] there exists $x_{0}\in X$ such that $\alpha \left(x_{0},Tx_{0}\right) \geq 1$ or $\eta \left( x_{0},Tx_{0}\right)\leq 1;$\\
\item[(iv)] $T$ is a $\left( \alpha ,\eta \right) -$continuous.\\

Then T has a fixed point. Moreover, $T$ has a unique fixed point when $\alpha\left( x,y\right) \geq 1$ or $\eta \left( x, y\right) \leq 1$ for all $x,y\in X.$\\
\end{theorem}
\begin{proof}
 Let $x_{0}\in X$ such that  $\alpha \left(x_{0},Tx_{0}\right) \geq 1$ or $\eta \left( x_{0},Tx_{0}\right) \leq 1.$\\
Define a sequence $\left\lbrace  x_{n}\right\rbrace  $ by $x_{n}=T^{n}x_{0}=Tx_{n-1}.$ Since T is a triangular $\left( \alpha,\eta \right) -$admissible mapping, then
 $\alpha \left( x_{0},x_{1}\right) =\alpha \left( x_{0},Tx_{0}\right) \geq 1\Rightarrow \alpha \left( Tx_{0},Tx_{1}\right) \geq 1=\alpha \left( x_{1},x_{2}\right) $ or $\eta \left( x_{0},x_{1}\right) =\eta \left( x_{0},Tx_{0}\right) \leq 1\Rightarrow \alpha \left( Tx_{0},Tx_{1}\right) \leq 1=\alpha \left( x_{1},x_{2}\right)$\\
Continuing this process we have $\alpha \left( x_{n-1},x_{n}\right) \geq 1$ or $\eta \left( x_{n-1},x_{n}\right) \leq 1,$
for all $n\in \mathbb{N}.$ By $\left( T_{3}\right) $ and $\left( T_{4}\right) ,$ one has.
\begin{equation}
\alpha \left( x_{m},x_{n}\right) \geq 1\text{ or }\eta \left(x_{m},x_{n}\right) \leq 1,\text{ \ \ \ \ \ }\forall m,n\in \mathbb{N},\text{ }m\neq n. 
\end{equation}
Suppose that there exists $n_{0}$ $\in \mathbb{N}$ such that $x_{n_{0}}=Tx_{n_{0}}.$ Then $x_{n_{0}}$ is a fixed point of $T$
and the prove is finished. Hence, we assume that $x_{n}\neq Tx_{n}$,
i.e. $d\left( x_{n-1},x_{n}\right) >0$ for all $\ n\in \mathbb{N}.$ We have
\begin{equation}
x_{n}\neq x_{m},\ \forall m,n\in \mathbb{N},m\neq n. 
\end{equation}
Indeed, suppose that $x_{n}=x_{m}$ for some $m=n+k>n,$ so we have\\
$$x_{n+1}=Tx_{n}=Tx_{m}=x_{m+1}.$$ \\
Denote $ d_{m}=d\left( x_{m},x_{m+1}\right).$
Then, $\left( 3.1\right) $ and lemma $\left( {2.5}\right) $ implies that
\begin{equation*}
\theta \left( d_{n}\right) =\theta \left( d_{m}\right)\leq\theta \left(s^{2}d_{m}\right) =\theta \left( s^{2}d\left( Tx_{m-1},Tx_{m}\right) \right)\leq \phi \left( \theta \left( \beta _{1}d_{m-1}+\beta _{2}d_{m-1}+\beta_{3}d_{m}\right) \right)
\end{equation*}
\begin{equation*}
<\theta \left( \beta _{1}d_{m-1}+\beta _{2}d_{m-1}+\beta _{3}d_{m}\right).\\
\end{equation*}
As $\theta $ is increasing, so 
$$d_{n}=d_{m}<\beta _{1}d_{m-1}+\beta_{2}d_{m-1}+\beta _{3}d_{m}.$$ 
Hence 
$$d_{m}<\frac{\beta _{1}+\beta _{2}}{1-\beta _{3}}d_{m-1}.$$\\
Since $$\beta _{1}+\beta _{2}+\beta _{3}\leq 1.$$\\ Thus $$d_{m}<d_{m-1}.$$ Continuing this process, we
can prove that $d_{n}=d_{m}<d_{n},$ which is a Contradiction. Thus, in  follows, we can assume that $\left( 3.2\right) $ and $\left( 3.3\right) $ hold.\\
We shall prove that 
\begin{equation}
\lim\limits_{n \rightarrow +\infty}d\left( x_{n},x_{n+1}\right) =0\text{ \ and \ }\lim\limits_{n \rightarrow +\infty}d\left( n_{n},x_{n+2}\right) =0.  
\end{equation}
Since $T$ $\ $is  $\left( \alpha ,\eta \right) -\theta -\phi -$ contraction, we get
\begin{align*}
   \theta \left( d_{n}\right)      &  = \theta \left( d\left( Tx_{n-1},Tx_{n}\right)\right) \\
                          & \leq \theta \left( s^{2}d\left( Tx_{n-1},Tx_{n}\right) \right)\\
                          & \leq \phi \left( \theta \left( \beta _{1}d_{n-1}+\beta _{2}d_{n-1}+\beta_{3}d_{n}\right) \right)\\
                          & <\theta \left( \beta _{1}d_{n-1}+\beta_{2}d_{n-1}+\beta _{3}d_{n}\right).\\
\end{align*}
Since $\theta $ is increasing, we deduce that $d_{n}<\beta _{1}d_{n-1}+\beta
_{2}d_{n-1}+\beta _{3}d_{n},$ thus
\begin{equation}
d_{n}<\frac{\beta _{1}+\beta _{2}}{1-\beta _{3}}d_{n-1}.   
\end{equation}
Since $\frac{\beta _{1}+\beta _{2}}{1-\beta _{3}}\leq 1.$ Then\\ 
\begin{equation}
d_{n}<d_{n-1}.
\end{equation}
Therefore $ d(x{_n},x_{n+1}) $ is monotone strictly decreasing sequence of non negative real numbers. Consequently, there exists $ \alpha \geq 0 $, such that such that
$$ \lim\limits_{n \rightarrow +\infty}d\left( x_{n},x_{n+1}\right) =\alpha .$$\\ 
which again by $\left( 3.1\right), \left( 3.6\right) $ and property of $\left( \theta \right),$ we have
\begin{equation*}
1<\theta \left( d_{n}\right) \leq \phi \left( \theta \left( \beta_{1}d_{n-1}+\beta _{2}d_{n-1}+\beta _{3}d_{n}\right) \right)  
\end{equation*}
\begin{equation}
\leq\phi \left( \theta \left( d_{n-1}\right) \right) \leq \phi ^{2}\left( \theta\left( d_{n-2}\right) \right) \leq...\leq\phi ^{n}\left( \theta \left(d_{0}\right) \right) =\phi ^{n}\left( \theta \left( d\left(x_{0},x_{1}\right) \right) \right). 
\end{equation}
By taking the limit as $n$ $\rightarrow $ $\infty $ in $\left( 3.7\right) $ and using $\left( \Phi _{2}\right) $ , we have 
\begin{equation*}
1\leq \lim\limits_{n \rightarrow +\infty}\theta \left( d\left( x_{n},x_{n+1}\right)\right) \leq \phi ^{n}\left( \theta \left( d\left( x_{0},x_{1}\right)\right) \right).\\
\end{equation*}
Then $\lim\limits_{n \rightarrow +\infty}\theta \left( d\left( x_{n},x_{n+1}\right)
\right) =1$, 
by $ \Theta_2,$ we obtain 
\begin{equation}
\lim\limits_{n \rightarrow +\infty}d\left( x_{n},x_{n+1}\right) =0.  
\end{equation}
On the other hand,\\
$\theta \left( s^{2}d\left( x_{n},x_{n+2}\right) \right) \leq \phi \left[ \theta \left( \beta _{1}d\left( x_{n-1},x_{n+1}\right) +\beta _{2}d\left(x_{n-1},x_{n}\right) +\beta _{3}d\left( x_{n+1},x_{n+2}\right) +\beta_{4}d\left( x_{n+1},x_{n}\right) \right) \right]  $ \\
$\leq \phi \left[  \theta \left( \begin{array}{c}
s\beta _{1}d\left( x_{n-1},x_{n+2}\right) +s\beta _{1}d\left(x_{n+2},x_{n}\right) +s\beta _{1}d\left( x_{n},x_{n+1}\right) +\beta_{2}d\left( x_{n-1},x_{n}\right)  \\ 
+\beta _{3}d\left( x_{n+1},x_{n+2}\right) +\beta _{4}d\left(x_{n+1},x_{n}\right) \end{array}\right)\right]  $\\
 $\leq \phi \left[  \theta \left( \begin{array}{c}s^{2}\beta _{1}d\left( x_{n-1},x_{n}\right) +s^{2}\beta _{1}d\left(x_{n},x_{n+1}\right) +s^{2}\beta _{1}d\left( x_{n+1},x_{n+2}\right) +s\beta_{1}d\left( x_{n+2},x_{n}\right)  \\ 
+s\beta _{1}d\left( x_{n},x_{n+1}\right) +\beta _{2}d\left(x_{n-1},x_{n}\right) +\beta _{3}d\left( x_{n+1},x_{n+2}\right) +\beta
_{4}d\left( x_{n+1},x_{n}\right) \end{array}\right) \right].$\\
By $ \theta_1 $ and Lemma $ 2.5$, we obtain
\begin{equation*}
s^{2}d\left( x_{n},x_{n+2}\right) <s^{2}\beta _{1}d\left(x_{n-1},x_{n}\right) +s^{2}\beta _{1}d\left( x_{n},x_{n+1}\right)+s^{2}\beta _{1}d\left( x_{n+1},x_{n+2}\right) +s\beta _{1}d\left(x_{n+2},x_{n}\right) 
\end{equation*}
\begin{equation*}
+s\beta _{1}d\left( x_{n},x_{n+1}\right) +\beta _{2}d\left(x_{n-1},x_{n}\right) +\beta _{3}d\left( x_{n+1},x_{n+2}\right) +\beta
_{4}d\left( x_{n+1},x_{n}\right),
\end{equation*}
Therefore,
\begin{equation*}
\left( s^{2}-s\beta _{1}\right) d\left( x_{n},x_{n+2}\right) <s^{2}\beta_{1}d\left( x_{n-1},x_{n}\right) +s^{2}\beta _{1}d\left(
x_{n},x_{n+1}\right) +s\beta _{1}d\left( x_{n+1},x_{n+2}\right) 
\end{equation*}
\begin{equation}
+s\beta _{1}d\left( x_{n},x_{n+1}\right) +\beta _{2}d\left(x_{n-1},x_{n}\right) +\beta _{3}d\left( x_{n+1},x_{n+2}\right) +\beta
_{4}d\left( x_{n+1},x_{n}\right). 
\end{equation}
Taking the limit as $\ n$ $\rightarrow \infty $ in $\left( 3.9\right) $ and using $\left( 3.8\right) ,$ since $ s^{2}-s\beta _{1}>0$ , we have
\begin{equation*}
\lim\limits_{n \rightarrow +\infty}d\left( x_{n},x_{n+2}\right) =0.\\
\end{equation*}
Hence $\left( 3.4\right) $ is proved.\\
Next, we show that $\left\lbrace  x_{n}\right\rbrace  $ is an Cauchy sequence in\ $X,$
if otherwise there exists an $\varepsilon $ $>0$ for which we can find sequences of positive integers \ $\left\lbrace  x_{n_{\left( k\right) }}\right\rbrace _k $ and $\left\lbrace x_{m_{\left( k\right) }}\right\rbrace _k $ of  $\left( x_{n}\right) $ such that,for all positive integers
$k$, $n_{\left( k\right) }>m_{\left( k\right) }>k,$
\begin{equation}
d\left(x_{m_{\left( k\right) }},x_{n_{\left( k\right) }}\right) \geq \varepsilon 
\end{equation}
and
\begin{equation}
\text{ }d\left( x_{m_{\left( k\right) }},x_{n_{\left( k\right) -1}}\right)
<\varepsilon   
\end{equation}
From $\left( 3.10\right) $ and using the rectangular inequality, we get
\begin{equation}
\varepsilon \leq d\left( x_{m_{\left( k\right) }},x_{n_{\left( k\right)}}\right) \leq sd\left( x_{m_{\left( k\right) }},x_{m_{\left( k\right)+1}}\right) +sd\left( x_{m_{\left( k\right) +1}},x_{n_{\left( k\right)+1}}\right) +sd\left( x_{n_{\left( k\right) +1}},x_{n_{\left( k\right)}}\right).\\
\end{equation}
Taking the upper limit as $k$ $\rightarrow \infty $ in $ (3.12) $ and using $\left(3.4\right),$ we get
\begin{equation}
\frac{\varepsilon }{s}\lim\limits_{n \rightarrow +\infty}\sup d\left(x_{m_{\left( k\right) +1}},x_{n_{\left( k\right) +1}}\right). 
\end{equation}
 Moreover, 
\begin{equation*}
d\left( x_{m_{\left( k\right) }},x_{n_{\left( k\right) }}\right) \leq sd\left( x_{m_{\left( k\right) }},x_{n_{\left( k\right) -1}}\right)+sd\left( x_{n_{\left( k\right) -1}},x_{n_{\left( k\right) +1}}\right)+sd\left( x_{n_{\left( k\right)+1}},x_{n_{\left( k\right) }}\right).\\
\end{equation*}
Then, from $\left( 3.11\right) $ and $\left( 3.8\right),$ we get
\begin{equation}
\lim\limits_{n \rightarrow +\infty}\sup d\left( x_{m_{\left( k\right) }},x_{n_{\left(
k\right) }}\right) \leq s\varepsilon. 
\end{equation}
On the other hand we have
\begin{equation}
d\left( x_{n_{\left( k\right) }},x_{m_{\left( k\right) +1}}\right) \leq sd\left( x_{n_{\left( k\right) }},x_{n_{\left( k\right) -1}}\right)+sd\left( x_{n_{\left( k\right) -1}},x_{m_{\left( k\right) }}\right)+sd\left( x_{x_{m_{\left( k\right) }}},x_{m_{\left( k\right) +1}}\right).
\end{equation}
Then , from $\left( 3.11\right) $ and $\left( 3.8\right) $, we get 
\begin{equation}
\lim\limits_{n \rightarrow +\infty}\sup d\left( x_{n_{\left( k\right) }},x_{m_{\left(
k\right) +1}}\right) \leq s\varepsilon.  
\end{equation}
Applying $\left( 3.1\right)$ with $ x=  x_{m_{\left( k\right) }}$ and $ y= x_{n_{\left( k\right) }} $, we have 
\begin{equation*}
\theta \left( s^{2}d\left( x_{m_{\left( k\right) +1}},x_{n_{\left( k\right)+1}}\right) \right) =\theta \left( s^{2}d\left( Tx_{m_{\left( k\right)}},Tx_{n_{\left( k\right) }}\right) \right) 
\end{equation*}
\begin{equation}
\leq \phi \left( \theta \left( \begin{array}{c}\beta _{1}d\left( x_{m_{\left( k\right) }},x_{n_{\left( k\right) }}\right)
+\beta _{2}d\left( x_{m_{\left( k\right) }},Tx_{m_{\left( k\right) }}\right) \\ 
+\beta _{3}d\left( x_{n_{\left( k\right) }},Tx_{n_{\left( k\right) }}\right)+\beta _{4}d\left( x_{n_{\left( k\right) }},Tx_{m_{\left( k\right) }}\right) \end{array}\right) \right). 
\end{equation}
Now taking the upper limit as $k\rightarrow $ $\infty $ in $\left(3.17\right) $ and using $\left( \theta _{1}\right) ,\left( \theta_{3}\right) ,\left( \phi _{3}\right) $,  $ \left(3.13\right), $ $\left( 3.8\right) ,\left(3.14\right) ,\left( 3.16\right) $ and Lemma $\left( {2.5}\right) $ we have
\begin{equation}
\theta \left( s^{2}.\frac{\varepsilon }{s}\right) =\theta \left( \varepsilon .s\right) \leq \theta \left( s^{2}\lim\limits_{k \rightarrow +\infty}\sup d\left( x_{m_{\left( k\right) +1}},x_{n_{\left( k\right) +1}}\right) \right) 
\end{equation}
\begin{equation*}
\leq \phi \left( \theta \left( \begin{array}{c}\beta _{1}\lim\limits_{k \rightarrow +\infty}\sup d\left( x_{m_{\left( k\right)
}},x_{n_{\left( k\right) }}\right) +\beta _{2}\lim\limits_{k \rightarrow +\infty}\sup
d\left( x_{m_{\left( k\right) }},Tx_{m_{\left( k\right) }}\right)  \\ 
+\beta _{3}\lim\limits_{k \rightarrow +\infty}\sup d\left( x_{n_{\left( k\right)
}},Tx_{n_{\left( k\right) }}\right) +\beta _{4}\lim\limits_{k \rightarrow +\infty}\sup
d\left( x_{n_{\left( k\right) }},Tx_{m_{\left( k\right) }}\right) \end{array}\right) \right)
\end{equation*}
\begin{equation}
\leq\phi \left( \theta \left( \beta _{1}s\varepsilon +\beta_{4}s\varepsilon \right) \right)  <\theta \left(s\varepsilon \left( \beta _{1}+\beta _{4}\right) \right).
\end{equation}
Therefore
$\varepsilon .s<s\varepsilon \left( \beta _{1}+\beta _{4}\right)$ implies $s<\beta _{1}+\beta _{4}.$ Which is a contradiction.\\ Consequently, $\left\lbrace  x_{n}\right\rbrace  $ is a Cauchy sequence in $ \alpha-\eta- $complete rectangular b- metric space $\left( X,d\right).$
Since $\alpha \left( x_{n-1},x_{n}\right) \geq 1$ or $\eta \left( x_{n-1},x_{n}\right)\leq 1,$ for all $n\in \mathbb{N}$.\\
 This implies that the sequence $\left\lbrace  x_{n}\right\rbrace  $ is converges to some $z\in X$.
Suppose that\\ $z\neq Tz$. Then, we have all the assumption of Lemma 2.2 and $T$ is $\left( \alpha ,\eta \right)-$continuous, then  $Tx_{n}\rightarrow Tz  $ as $ n\rightarrow \infty$.
Therefore,
\begin{equation*}
\frac{1}{s}d\left( z,Tz\right) \leq\lim\limits_{n \rightarrow +\infty}\sup d\left(
x_{n},Tx_{n}\right) =0.
\end{equation*}
Hence we have $d\left( z,Tz\right) =0$ and so $Tz=z.$ Thus $z$ is a fixed point of $T.$\\
Uniqueness.\\Let $z,u\in $ Fix $\left( T\right) $ where $z\neq u$ and $\alpha \left(z,u\right) \geq 1$ or $\eta \left( z,u\right) \leq 1.$\\
Applying $\left(3.1\right) $ with $ x=z $ and $ y=u $ we have
\begin{align*}
\theta \left( d\left( z,u\right) \right)  &=\theta \left( d\left(Tz,Tu\right) \right)\\
&\leq \theta \left( s^{2}d\left( Tz,Tu\right) \right)\\ 
&\leq \phi \left( \theta \left( \beta _{1}d\left( z,u\right) +\beta_{2}d\left( z,Tz\right) +\beta _{3}d\left( u,Tu\right) +\beta _{4}d\left(Tz,u\right) \right) \right)\\
&\leq \phi \left( \theta \left( \beta _{1}d\left( z,u\right)  +\beta _{4}d\left(Tz,u\right) \right) \right)\\
& \leq \phi \left( \theta \left( d\left( z,u\right) \right) \right). 
\end{align*}
Since $ \theta $ is increasing, Therefore
\begin{equation*}
d\left( z,u\right) <d\left( z,u\right).
\end{equation*}
Which is a contradiction. Hence, $z=u$ and $T$ has a unique fixed point.
\end{proof}
 Recall that a self-mapping $ T $ is said to have the property $ P $, if $ Fix(T)=Fix( T^n) $ for every $ n\in \mathbb{N} $  
\begin{theorem}
 \bigskip Let $\alpha ,\eta $ : $X\times X\rightarrow \mathbb{R}^{+}$ be two function and let $\left( X,d\right) $ be an $\left( \alpha,\eta \right) -$complete rectangular b-metric space. Let $T:X\rightarrow X$ be a mapping satisfying the following conditions:\\
\item[(i)] $T$ is a triangular $\left( \alpha ,\eta \right) -$ admissible mapping;\\
\item[(ii)] $T$ is an $\left( \alpha ,\eta \right) -\theta -\phi -$contraction;\\
\item[(iii)] $\alpha \left( z,Tz\right) \geq 1$ or $\eta \left(z,Tz\right) \leq 1,$ for all $z\in $ Fix $\left( T\right) .$\\
Then $\ T$ has the property $P.$\\
\end{theorem}
\begin{proof}
Let $z$ $\in $ Fix $\left( T^{n}\right) $ for some fixed $n>1$. As $\alpha \left( z,Tz\right) \geq 1$ or $\eta \left( z,Tz\right) \leq 1$ \\and  $T$ is a triangular $\left( \alpha ,\eta \right) $-admissible mapping, then\\ 
\begin{equation*}
\alpha \left( Tz,T^{2}z\right) \geq1 \text{ or }  \eta \left( T^{2}z,Tz\right)\leq 1. 
\end{equation*}
Continuing this process, we have
\begin{equation*}
\alpha \left( T^{n}z,T^{n+1}z\right) \geq 1\text{ or }\eta \left(T^{n}z,T^{n+1}z\right) \leq 1,
\end{equation*}
for all $n\in \mathbb{N}$. By $\left( T_3\right)$ and $\left( T_4\right)$  we get
\begin{equation*}
\alpha \left( T^{m}z,T^{n}z\right) \geq 1\text{ or }\eta \left(
T^{m}z,T^{n}z\right) \leq 1,\text{ \ }\forall \text{ }m,n\in \mathbb{N},\text{ }n\neq m.
\end{equation*}
Assume that $z\notin $ Fix $\left( T\right) ,$i.e. $d\left( z,Tz\right) >0.$\\ 
Applying $\left(3.1\right) $ with $ x=T^{n-1}z $ and $ y=z $, we get
\begin{equation*}
d\left( z,Tz\right) =d\left( T^{n}z,Tz\right) =d\left( TT^{n-1}z,Tz\right)\leq s^{2}d\left(TT^{n-1}z,Tz\right),
\end{equation*}
which implies that
\begin{align*}
\theta(d\left(TT^{n-1}z,Tz\right))&\leq \phi\left[ \theta\left( \beta _{1}d\left( T^{n-1}z,z\right) +\beta _{2}d\left(T^{n-1}z,T^{n}z\right) +\beta _{3}d\left( z,Tz\right) +\beta _{4}d\left(
z,T^{n}z\right)\right) \right]\\ 
&< \theta\left[ \beta _{1}d\left( T^{n-1}z,z\right) +\beta _{2}d\left(T^{n-1}z,T^{n}z\right) +\beta _{3}d\left( z,Tz\right) +\beta _{4}d\left(z,T^{n}z\right)\right]\\ 
&=\theta\left[ \beta _{1}d\left( T^{n-1}z,T^{n}z\right) +\beta _{2}d\left(T^{n-1}z,T^{n}z\right) +\beta _{3}d\left( z,Tz\right)\right]. 
\end{align*}
Since $ \theta $ is increasing, Therefore,
\begin{equation*}
 d\left( z,Tz\right) <\frac{\beta _{1}+\beta _{2}}{1-\beta _{3}}d\left( T^{n-1}z,T^{n}z\right) \leq d\left(T^{n-1}z,T^{n}z\right),
\end{equation*}
which is a contradiction as $d\left( T^{n-1}z,T^{n}z\right) \rightarrow 0$ and $d\left( z,Tz\right) >0.$
\end{proof}
Assuming the following conditions, we prove that Theorem $ 3.3 $ still hold for $ T $ not necessarily continuous: 
\begin{theorem}
Let $\alpha,\eta :$ $X\times X\rightarrow \mathbb{R} ^{+}$ be two functions and let $d\left( X,d\right) $ be an $\left( \alpha,\eta \right) -$complete rectangular b-metric space.\\
Let $T:X\rightarrow X$ be a mapping satisfying the following assertions:\\
\item[(i)]$T$ is triangular $\left( \alpha ,\eta \right) -$admissible;\\
\item[(ii)] $T$ is $\left( \alpha ,\eta \right) -\theta
-\phi -$contraction;\\
\item[(iii)] there exists $x_{0}\in X$ such that $\alpha \left(x_{0},Tx_{0}\right) \geq 1$ or $\eta \left( x_{0},Tx_{0}\right) \leq 1;$\\
\item[(iv)] $\left( X,d\right) $ is an $\left( \alpha ,\eta \right) $-regular rectangular b-metric space.\\

Then $T$ has a fixed point. Moreover, $T$ has a unique fixed point whenever $\\
\alpha \left( z,u\right) \geq 1$ or $\eta \left( z,u\right) \leq 1$ for all $z,u\in Fix\left( T\right).$
\end{theorem}
\begin{proof}
Let x$_{0}\in X$ such that $\alpha \left(x_{0},Tx_{0}\right) \geq 1$ or $\eta \left( x_{0},Tx_{0}\right) \leq 1$.
 Similar to the proof of Theorem 3.3, we can conclude that 
\begin{equation*}
\left( \alpha \left( x_{n},x_{n+1}\right) \geq 1\text{ or }\eta \left(x_{n},x_{n+1}\right) \leq 1\right),\text{ and }x_{n}\rightarrow z\text{ as }n\rightarrow \infty,
\end{equation*}
where $x_{n+1}=Tx_{n}.$ \\
From $\left( iv\right)$ $ \alpha \left( x_{n+1},z\right) \geq 1\text{ or }\eta \left( x_{n+1},z\right)\leq 1 $, hold for $n\in \mathbb{N}.$\\
Suppose that $Tz=x_{n_{0+1}}=Tx_{n_{0}}$ for some $n_{0}\in \mathbb{N}.$ From the  Theorem 3.3 we know that the members of the sequence $\left\lbrace x_{n}\right\rbrace  $ are distinct. Hence, we have $Tz\neq Tx_{n},$ i.e. $d\left(Tz,Tx_{n}\right) >0$ for all $n>n_{0}.$ Thus, we can apply $\left(3.1\right) $, to $x_{n}$ and $z$ for all $n>n_{0}$ to get
\begin{align*}
\theta \left( d\left( Tx_{n},Tz\right) \right) &\leq \theta \left(s^{2}d\left( Tx_{n},Tz\right) \right)\\
& \leq \phi \left( \theta \left( \beta_{1}d\left( x_{n},z\right) +\beta _{2}d\left( x_{n},Tx_{n}\right) +\beta_{3}d\left( z,Tz\right) +\beta _{4}d\left( z,Tx_{n}\right) \right) \right).
\end{align*}
By Lemma $\left( {2.5}\right) $ and $\left( \theta _{1}\right) $, we obtain
\begin{equation}
d\left( Tx_{n},Tz\right) <\left( \beta _{1}d\left( x_{n},z\right) +\beta_{2}d\left( x_{n},Tx_{n}\right) +\beta _{3}d\left( z,Tz\right) +\beta_{4}d\left( z,Tx_{n}\right) \right).
\end{equation}
By taking the limit as $n$ $\rightarrow $ $\infty $ in $\left( 3.20\right) $, we have 
$$ \lim_{n\rightarrow \infty }\sup d\left(Tx_{n},Tz\right) \leq \beta _{3}d\left( z,Tz\right). $$
Assume that $z\neq Tz.$ Then, from lemma 2.2, 
\begin{equation*}
\frac{1}{s}d\left( z,Tz\right) \leq \lim_{n\rightarrow \infty }\sup d\left(
Tx_{n},Tz\right) \leq \beta _{3}d\left( z,Tz\right).
\end{equation*}
by assumption $\beta _{3}<\frac{1}{s},$ we have $d\left( z,Tz\right) =0
$ and so $z=Tz.$ Thus, $z$ is a fixed point of $T.$\\
The proof of the uniqueness is similarly to that of Theorem 3.3.\\
\end{proof}
Above Theorems. If we take $\phi \left( t\right) =t^{k},$ for some fixed $k\in %
\left] 0,1\right[ ,$ \\
where $\beta _{1}=1\	and \ \beta _{2}=\beta _{3}=\beta _{4}=0.$
We obtain the following extension of  Jamshaid et al, result $\left(\text{Theorem 2.2}\right)$  \cite{JA}  of $\left( \alpha ,\eta \right) -$ complete rectangular $b$-metric space.
\begin{corollary}
Let $\alpha ,\eta :X\times X\rightarrow \left[ 0,+\infty \right[ $ be two functions, $d\left( X,d\right) $ be an $\left( \alpha ,\eta \right) -$ complete rectangular $b$-metric space and let $T:X\rightarrow X$ be self-mapping. Suppose for all $x,y\in X$ with $\alpha \left( x,y\right) \geq 1$ or $\eta \left( x,y\right) \leq 1$ and $d\left( Tx,Ty\right) >0$\\
 we have \\
\begin{equation*}
\theta\left[ s^{2} d\left( Tx,Ty\right) \leq \left[ \theta \left( d\left( x,y\right)\right] 
\right) \right] ^{k},
\end{equation*}
where $\ \theta \in \Theta $ and $k\in \left] 0,1\right[$. If the mapping $ T $ satisfying the following assertions:\\
point, if \\
\item[(i)] $T$ is a triangular $\left( \alpha ,\eta \right)-$admissible mapping;\\
\item[(ii)] there exists $x_{0}\in X$ such that $\alpha \left(x_{0},Tx_{0}\right) \geq 1$ or $\eta \left( x_{0},Tx_{0}\right) \leq 1;$\\
\item[(iii)] $T$ is $\left( \alpha,\eta \right) $ continuous; or \\
\item[(iv)] is an $\left( \alpha ,\eta \right) -$regular rectangular $b-$metric space.\\
Then $T$ has a fixed point. Moreover, $T$ has a unique fixed point whenever $\\
\alpha \left( z,u\right) \geq 1$ or $\eta \left( z,u\right) \leq 1$ for all $z,u\in Fix\left( T\right).$
\end{corollary}
\begin{proof}
Let $\phi \left( t\right)$ $=t^{k }$, we prove that  $T$ is an $\left( \alpha ,\eta \right)
-\theta -\phi -$contraction, Hence T satisfies in assumption of Theorem $ 3.3 $ or Theorem $ 3.5 $ and is the unique fixed point of $%
T.$\\
\end{proof}
It follows from Theorem 3.3, we obtain the following fixed point theorems
for $\theta -\phi $- Kannan-type contraction and $\theta -\phi $ Reich-type contraction.\\
\begin{theorem}
Let $\left( X,d\right) $ be a $\left( \alpha ,\eta \right) -$complete rectangular $b$-metric space and let $\alpha ,\eta :X\times X\rightarrow \left[ 0,+\infty \right[ $ be two functions. Let $T:X\times X\rightarrow X$ be a self mapping satisfying following conditions:\\
\item[(i)]  $T$ is a triangular $\left( \alpha ,\eta \right) -$ admissible mapping;\\
\item[(ii)] $T$ is a $\left( \alpha ,\eta \right) -\theta -\phi -$Kannan-type contraction;\\
\item[(iii)] $T$ there exists $x_{0}\in X$ such that $\alpha \left(x_{0},Tx_{0}\right) \geq 1$ or $\eta \left( x_{0},Tx_{0}\right) \leq 1;$\\
\item[(iv)] $T$is a $\left( \alpha ,\eta \right) -$continuous.\\
Then T has a fixed point. Moreover, $T$ has a unique fixed point when $ \alpha \left( x,y\right) \geq 1$ or $\eta \left( x,y\right) \leq 1$ for all $x,y\in X.$\\
\end{theorem}
\begin{proof}
If $T$ is a $\left( \alpha ,\eta \right) -\theta -\phi -$ Kannan-type contraction, thus 
there exists $\theta \in \Theta $ and $\phi \in \Phi $ with $\alpha \left( x,y\right) \geq 1$ or $\eta \left(x,y\right) \leq 1$ for any $x,y\in X,$ $Tx\neq Ty$, we have\\ 
\begin{equation*}
\theta \left[ s^{2}d\left( Tx,Ty\right) \right] \leq \phi \left[ \theta\left(\frac{\left( d\left( Tx,x\right) +d\left( Ty,y\right) \right) }{2s}\right) \right].\\
\end{equation*}
Therefore,
\begin{equation*}
\theta \left[ s^{2}d\left( Tx,Ty\right) \right] \leq \phi \left[ \theta\left( ( \beta _{1}d\left( x,y\right) + \beta _{2}d\left( Tx,x\right)+\beta _{3}d\left( Ty,y\right)+ \beta _{4}d\left( Tx,y\right)\right) \right],\\
\end{equation*}
where $\beta _{1}=\beta _{4}=0,$ $\beta _{2}=\beta _{3}=\frac{1}{2s},$ which implies that $T$ is a $\left( \alpha ,\eta \right) -\theta -\phi $ contraction 
Therefore, from the Theorem $\left( {3.3}\right).$ $T$ has a unique fixed point.\\
\end{proof}
\begin{theorem}
Let $\left( X,d\right) $ be a $\left( \alpha ,\eta \right) -$complete rectangular $b$-metric space and let $\alpha ,\eta :X\times X\rightarrow \left[ 0,+\infty \right[ $ be two functions. Let $T:X\times X\rightarrow X$ be a self mapping satisfying following conditions:\\
\item[(i)] $T$ is a triangular $\left( \alpha ,\eta \right)-$ admissible mapping;\\
\item[(ii)] $T$ is a $\left( \alpha ,\eta \right) -\theta -\phi -$ Reich-type contraction;\\
\item[(iii)] $T$ there exists $x_{0}\in X$ such that $\alpha \left(x_{0},Tx_{0}\right) \geq 1$ or $\eta \left( x_{0},Tx_{0}\right) \leq 1;$\\
\item[(iv)] $T$ is a $\left( \alpha ,\eta \right) -$continuous.\\
Then T has a fixed point. Moreover, $T$ has a unique fixed point when\\ $\alpha \left( x,y\right) \geq 1$ or $\eta \left( x,y\right) \leq 1$ for all $x,y\in X.$\\
\end{theorem}
\begin{proof}
If $T$ is a $\left( \alpha ,\eta \right) -\theta -\phi -$ Reich-type contraction, thus 
there exists $\theta \in \Theta $ and $\phi \in \Phi $ with $\alpha \left( x,y\right) \geq 1$ or $\eta \left(x,y\right) \leq 1$ for any $x,y\in X,$ $Tx\neq Ty$, we have 
$$\theta \left[s^{2} d\left( Tx,Ty\right) \right] \leq \phi \left[ \theta\left(\frac{\left( d\left(x,y\right)+ d\left( Tx,x\right) +d\left( Ty,y\right) \right) }{3s}\right) \right].$$
Therefore,
\begin{equation*}
\theta \left[ s^{2}d\left( Tx,Ty\right) \right] \leq \phi \left[ \theta\left( 
( \beta _{1}d\left( x,y\right) + \beta _{2}d\left( Tx,x\right)+\beta _{3}d\left( Ty,y\right)+ \beta _{4}d\left( Tx,y\right)\right) \right],
\end{equation*}
where $\beta _{1}=$ $\beta _{2}=\beta _{3}=\frac{1}{3s}\ and \ \beta _{4}=0$, which implies that $T$ is a $\left( \alpha ,\eta \right) -\theta -\phi $ contraction. Therefore, from the Theorem $\left( {3.3}\right).$ $T$ has a unique fixed
point.\\
\end{proof}
\begin{corollary}
Let $\left( X,d\right) $ be a $\left( \alpha ,\eta \right) -$complete rectangular $b$-metric space and let $\alpha ,\eta :X\times X\rightarrow \left[ 0,+\infty \right[ $ be two functions. Let $T:X\times X\rightarrow X$ be a self mapping satisfying following conditions:\\
\item[(i)] $T$ is a triangular $\left( \alpha ,\eta \right) -$ admissible mapping;\\
\item[(ii)] $T$ is a $\left( \alpha ,\eta \right) -$ Kannan type mapping,\\
\item[(iii)] $T$ there exists $x_{0}\in X$ such that $\alpha \left(x_{0},Tx_{0}\right) \geq 1$ or $\eta \left( x_{0},Tx_{0}\right) \leq 1;$\\
\item[(iv)] $T$is a $\left( \alpha ,\eta \right) -$continuous.\\

Then T has a fixed point. Moreover, $T$ has a unique fixed point when\\ $\alpha \left( x,y\right) \geq 1$ or $\eta \left( x,y\right) \leq 1$ for all $x,y\in X.$ Then $T$ has a unique fixed point $x\in X.$
\end{corollary}
\begin{proof}
Let $\theta (t)=$ $e^{t}$ for all $t$ $\in $ $\left] 0,+\infty \right[ $, and $\phi \left( t\right) =t^{2s\alpha }$ for all $t$ $\in \left[ 1,+\infty 
\right[ $.\\ It is obvious that $\theta \in \Theta$ and  $\phi \in \Phi$. We prove that $T$ is a $\theta -\phi $- Kannan-type contraction.\\ 
\begin{align*}
\theta \left(s^{2} d\left( Tx,Ty\right) \right)&=e^{\displaystyle s^{2}d\left( Tx,Ty\right)}\\
&\leq e^{\displaystyle \alpha \left( d\left( Tx,x\right) +d\left( y,Ty\right) \right)}\\
&=\displaystyle {e}^{\displaystyle 2s\alpha \left( \frac{d\left( Tx,x\right) +d\left( y,Ty\right) }{2s}\right) }\\
&=\left[ e^{\displaystyle \left( \frac{d\left( Tx,x\right) +d\left( y,Ty\right) }{2s}\right) }\right] ^{2s\alpha }\\
&=\phi \left[ \theta  \left( \frac{d\left( Tx,x\right) +d\left(y,Ty\right) }{2s} \right) \right].
\end{align*}
Therefore, from Theorem 3.3, $T$ has a unique fixed point $x\in X$\\
\end{proof}
\begin{corollary}
Let $\left( X,d\right) $ be a $\left( \alpha ,\eta \right) -$complete rectangular $b$-metric space and let $\alpha ,\eta :X\times X\rightarrow \left[ 0,+\infty \right[ $ be two functions. Let $T:X\times X\rightarrow X$ be a self mapping satisfying following conditions:\\
\item[(i)] $T$ is a triangular $\left( \alpha ,\eta \right) -$admissible mapping;\\
\item[(ii)] $T$ is a $\left( \alpha ,\eta \right)-$Reich type mapping,\\
\item[(iii)] $T$ there exists $x_{0}\in X$ such that $\alpha \left(x_{0},Tx_{0}\right) \geq 1$ or $\eta \left( x_{0},Tx_{0}\right) \leq 1;$\\
\item[(iv)] $T$is a $\left( \alpha ,\eta \right) -$continuous.\\

Then T has a fixed point. Moreover, $T$ has a unique fixed point when \\$\alpha \left( x,y\right) \geq 1$ or $\eta \left( x,y\right) \leq 1$ for all $x,y\in X.$\\
\end{corollary}
\begin{proof}
Let $\theta (t)=$ $e^{t}$ for all t $\in $ $\left] 0,+\infty \right[ $, and $\phi \left( t\right) =t^{3s\lambda }$
for all t $\in \left[ 1,+\infty \right[ $.\\ We prove that T is a $\theta -\phi $- Reich-type contraction. 
\begin{align*}
\theta \left( s^{2}d\left( Tx,Ty\right) \right)& =e^{\displaystyle s^{2}d\left( Tx,Ty\right) }\\
&\leq \displaystyle e^{\displaystyle3\displaystyle  \lambda s \frac{\left( d\left( x,y\right) +d\left( Tx,x\right) +d\left(y,Ty\right) \right) }{3s}}\\
&=\left[\displaystyle e^{\displaystyle  \frac{\left( d\left( x,y\right) +d\left( Tx,x\right)
+d\left( y,Ty\right) \right) }{3s}}\right] ^{3\lambda s }
\end{align*}
\begin{equation*}
=\phi \left[ \theta \left( \left( \frac{\left( d\left( x,y\right) +d\left(
Tx,x\right) +d\left( y,Ty\right) \right) }{3}\right) \right) \right] .
\end{equation*}
Therefore, from Theorem 3.3, $T$ has a unique fixed point $x\in X.$\\
\end{proof}
\begin{example}
	\bigskip Consider the set $X=\left\{ 1,2,3,4\right\} .$ It is easy to check
	that the mapping $d:X\times X\rightarrow \left[ 0,+\infty \right[ $ given by\\
	\item[(i)] $d\left( x,y\right) =d\left( y,x\right) ,$ $d\left(
	x,x\right) =0$ for all $\ x,y\in X,$\\
	\item[(ii)] $d\left( 1,2\right) =\frac{1}{24},$ $d\left( 1,3\right)
	=3,$ $d\left( 1,4\right) =4,$\\
	\item[(iii)] $d\left( 2,3\right) =5,$ $d\left( 2,4\right) =6,$ $%
	d\left( 3,4\right) =18.$\\
	
	Clearly $\left( X,d\right) $ is a rectangular $b-$metric space with parameter $%
	s=2.$\\
	
	Define mapping $T:X\rightarrow X$ and $\alpha ,\eta :X\times X\rightarrow \left[ 0,+\infty \right[ $ by\\
	\begin{equation*}
	\left\lbrace
	\begin{aligned}
	T(1)	&=1\\
	T(2)	&=1\\
	T(3)	&=1\\
	T(4)	&=2
	\end{aligned}
	\right.
	\end{equation*}
	and 
	$$ \alpha \left( x,y\right) =\frac{x+y}{\max \left\{ x,y\right\} },$$ \ $$%
	\eta \left( x,y\right) =\frac{\left\vert x-y\right\vert }{\max \left\{x,y\right\} }.$$\\
	
	Then, $T$ is an $\left( \alpha ,\eta \right) -$continuous triangular $\left(\alpha ,\eta \right) -$admissible mapping.\\
	
	Let $\theta \left( t\right) =\sqrt{t}+1,$ $\phi \left( t\right) =\frac{2t+1}{3}$ and $\beta _{1}=\frac{4}{10},$ $\beta _{2}=\frac{1}{10},$ $\beta _{3}=\frac{3}{10},$ $\beta _{4}=\frac{2}{10}.$\\
	It obvious that $\theta \in \Theta $ and $\phi \in\Phi .$\\
	Evidently, ($\left( \alpha \left( x,y\right) \geq 1\text{ or }\left( x,y\right) \leq 1\right) $ and $d\left( Tx,Ty\right) >0$ are when\\ 
	$\left\{ x,y\right\} =\left\{ 1,4\right\} ,\left\{ x,y\right\} =\left\{2,4\right\} ,$ or $\left\{ x,y\right\} =\left\{ 3,4\right\} $. Consider the
	following four possibilities:\\
	For $x=1,$ $y=4.$ then 
	\begin{equation*}
	\theta \left( s^{2}.d\left( T1,T4\right) \right) =\sqrt{\frac{1}{6}}+1= 1.4,\\
	\end{equation*}
	and%
	\begin{equation*}
	\phi \left( \theta \left( \beta _{1}d\left( 1,4\right) +\beta _{2}d\left(
	1,T1\right) \right) +\beta _{3}d\left( 4,T4\right) +\beta _{4}d\left(
	4,T1\right) \right) \\
	=\phi \left( \theta \left( \frac{21}{5}\right) \right) =2.36\text{,}\\
	\end{equation*}
	then 
	\begin{equation*}
	\theta \left( s^{2}.d\left( T1,T4\right) \right) \leq \phi \left( \theta
	\left( \beta _{1}d\left( 1,4\right) +\beta _{2}d\left( 1,T1\right) \right)
	+\beta _{3}d\left( 4,T4\right) +\beta _{4}d\left( 4,T1\right) \right).\\
	\end{equation*}
	
	For $x=2,$ $y=4.$ then
	
	\begin{equation*}
	\theta \left( s^{2}.d\left( T2,T4\right) \right) =\sqrt{\frac{1}{6}}+1=1.4,\\
	\end{equation*}
	
	and%
	\begin{eqnarray*}
		&&\phi \left( \theta \left( \beta _{1}d\left( 2,4\right) +\beta _{2}d\left(
		2,T2\right) \right) +\beta _{3}d\left( 4,T4\right) +\beta _{4}d\left(
		4,T2\right) \right) \\
		&=&\phi \left( \theta \left( 3.65\right) \right) =2.27\text{,}\\
	\end{eqnarray*}
	
	then 
	\begin{equation*}
	\theta \left( s^{2}.d\left( T2,T4\right) \right) \leq \phi \left( \theta
	\left( \beta _{1}d\left( 2,4\right) +\beta _{2}d\left( 2,T2\right) \right)
	+\beta _{3}d\left( 4,T4\right) +\beta _{4}d\left( 4,T2\right) \right).\\
	\end{equation*}
	
	For $x=3,$ $y=4.$ then
	
	\begin{equation*}
	\theta \left( s^{2}.d\left( T3,T4\right) \right) =\sqrt{\frac{1}{6}}+1=1.4,\\
	\end{equation*}
	
	and
	\begin{equation*}
	\phi \left( \theta \left( \beta _{1}d\left( 3,4\right) +\beta _{2}d\left(
	3,T3\right) \right) +\beta _{3}d\left( 4,T4\right) +\beta _{4}d\left(
	4,T3\right) \right) \\
	=\phi \left( \theta \left( 10.1\right) \right) =3.13\text{,}\\
	\end{equation*}
	then 
	\begin{equation*}
	\theta \left( s^{2}.d\left( T3,T4\right) \right) \leq \phi \left( \theta
	\left( \beta _{1}d\left( 3,4\right) +\beta _{2}d\left( 3,T3\right) \right)
	+\beta _{3}d\left( 4,T4\right) +\beta _{4}d\left( 4,T3\right) \right) .\\
	\end{equation*}
	Hence $  T $ satisfying assumption of Theorem (3.3) and 1 is the unique fixed point of $ T $.\\
\end{example}
\begin{example}
	Let $ X=A\cup B $, where $ A=\lbrace \frac{1}{n}:n\in\lbrace 2,3,4,5,6,7\rbrace \rbrace $ and $ B=\left[1,2 \right]  $. Define $ d:X\times X\rightarrow \left[0,+\infty \right[  $ as follows:
	\begin{equation*}
	\left\lbrace
	\begin{aligned}
	d(x, y) &=d(y, x)\ for \ all \  x,y\in X;\\
	d(x, y) &=0\Leftrightarrow y= x.\\	
	\end{aligned}
	\right.
	\end{equation*}
	and
	\begin{equation*}
	\left\lbrace
	\begin{aligned}		    
	d\left( \frac{1}{2},\frac{1}{3}\right) =d\left( \frac{1}{4},\frac{1}{5}\right) =d\left( \frac{1}{6},\frac{1}{7}\right) 	&=0.05\\
	d\left( \frac{1}{2},\frac{1}{4}\right) =d\left( \frac{1}{3},\frac{1}{7}\right) =d\left( \frac{1}{5},\frac{1}{6}\right) 	&=0.08\\
	d\left( \frac{1}{2},\frac{1}{6}\right) =d\left( \frac{1}{3},\frac{1}{4}\right) =d\left( \frac{1}{5},\frac{1}{7}\right) 	&=0.4\\
	d\left( \frac{1}{2},\frac{1}{5}\right) =d\left( \frac{1}{3},\frac{1}{6}\right) =d\left( \frac{1}{4},\frac{1}{7}\right) 	&=0.24\\
	d\left( \frac{1}{2},\frac{1}{7}\right) =d\left( \frac{1}{3},\frac{1}{5}\right) =d\left( \frac{1}{4},\frac{1}{6}\right) 	&=0.15\\
	d\left( x,y\right) =\left( \vert x-y\vert\right) ^{2} \ otherwise.
	\end{aligned}
	\right.
	\end{equation*}
	Then $ (X,d) $ is a rectangular b-metric space with coefficient s=3. However we have the following:
	\item[1)] $ (X,d) $ is not a  metric space, as $d\left( \frac{1}{5},\frac{1}{7}\right)=0.4>0.29=d\left( \frac{1}{5},\frac{1}{4}\right)+d\left( \frac{1}{4},\frac{1}{7}\right) 	$.
	\item[2)] $ (X,d) $ is not a  b-metric space for s=3, as $d\left( \frac{1}{3},\frac{1}{4}\right)=0.4>0.39=3\left[ d\left( \frac{1}{3},\frac{1}{2}\right)+d\left( \frac{1}{2},\frac{1}{4}\right)\right]  	$.  
	\item[3)] $ (X,d) $ is not a rectangular metric space, as $d\left( \frac{1}{5},\frac{1}{7}\right)=0.4>0.28=d\left( \frac{1}{5},\frac{1}{4}\right)+d\left( \frac{1}{4},\frac{1}{2}\right)+d\left( \frac{1}{2},\frac{1}{7}\right)$.\\ 
	Define mapping $T:X\rightarrow X$ and $\alpha ,\eta :X\times X\rightarrow \left[ 0,+\infty \right[ $ by\\
	\begin{equation*}
	T(x)=\left\lbrace
	\begin{aligned}
	\sqrt[6]{x}	& \ if \ x\in \left[1,2 \right]  \\
	& 1  \ if \ x\in A.\\
	\end{aligned}
	\right.
	\end{equation*}
	and 
	\begin{equation*}
	\alpha(x,y)=\left\lbrace
	\begin{aligned}
	& sinh(x+y)	 \ if \ x,y\in \left[1,2 \right]  \\
	& \frac{1}{e^{x+y}}	\ otherwise,\\
	\end{aligned}
	\right.
	\end{equation*}
	\begin{equation*}
	\eta(x,y)=\left\lbrace
	\begin{aligned}
	&\frac{x+y}{4}	 \ if \ x,y\in \left[1,2 \right]  \\
	& 1+e^{-(x+y)}	 \ otherwise. \\
	\end{aligned}
	\right.
	\end{equation*}
	Then, $T$ is an $\left( \alpha ,\eta \right) -$continuous triangular $\left(\alpha ,\eta \right) -$admissible mapping.\\
	Let $\theta \left( t\right) =\sqrt{t}+1,$ $\phi \left( t\right) =\frac{t+1}{2}$ and taking $\beta _{1}=1,$ $\beta _{2}=\beta _{3}=\beta _{3}=0.$ It obvious that $\theta \in \Theta $ and $\phi \in\Phi .$\\
	Evidently, ($\left( \alpha \left( x,y\right) \geq 1\text{ or }\left( x,y\right) \leq 1\right) $ and $d\left( Tx,Ty\right) >0$ are when $\left\{ x,y\right\}\in \left[1,2 \right]  $ with $x\neq y$.\\
	Consider two cases :\\
	case 1: $ x>y $. 
	$$ \theta(s^{2} d(Tx,Ty)= 3\left( \sqrt[6]{x}-\sqrt[6]{y}\right) +1$$
	and
	$$ \phi\left[ \theta( d(x,y))\right] =\frac{x-y}{2}+1.$$
	On the other hand 
	\begin{align*}
	\theta(s^{2} d(Tx,Ty)- \phi\left[ \theta( d(x,y))\right]& =\frac{6\left( \sqrt[6]{x}-\sqrt[6]{y}\right)-( x-y)}{2}\\
	&=\frac{1}{2}\left(\left( \sqrt[6]{x}-\sqrt[6]{y}\right) \right)\left[  6-\left( \sqrt[6]{x{^5}}+\sqrt[6]{x{^4}y}+\sqrt[6]{x{^3}y{^2}}+\sqrt[6]{x{^2}y{^3}}+\sqrt[6]{xy{^4}}+\sqrt[6]{y{^5}}\right) \right].  
	\end{align*}
	Since $ x,y\in \left[1,2 \right]  $, then
	$$\left[  6-\left( \sqrt[6]{x{^5}}+\sqrt[6]{x{^4}y}+\sqrt[6]{x{^3}y{^2}}+\sqrt[6]{x{^2}y{^3}}+\sqrt[6]{xy{^4}}+\sqrt[6]{y{^5}}\right) \right]\leq 0  .$$
	Which implies that
	\begin{align*}
	\theta(s^{2} d(Tx,Ty) &\leq  \phi\left[ \theta( d(x,y))\right]\\
	&= \phi\left[ \theta(\beta _{1} d(x,y))+\beta _{2} d(x,Tx)+\beta _{3} d(y,Ty)+\beta _{4} d(y,Tx)\right] 
	\end{align*}
	case 2: $ y>x $. 
	$$ \theta(s^{2} d(Tx,Ty)= 3\left( \sqrt[6]{y}-\sqrt[6]{x}\right) +1$$
	and
	$$ \phi\left[ \theta( d(x,y))\right] =\frac{y-x}{2}+1.$$
	Similarly of case 2, we conclude that
	\begin{align*}
	\theta(s^{2} d(Tx,Ty) &\leq \phi\left[ \theta(\beta _{1} d(x,y))+\beta _{2} d(x,Tx)+\beta _{3} d(y,Ty)+\beta _{4} d(y,Tx)\right] 
	\end{align*}
	Hence, the condition $ (3.1) $ is satisfied. Therefore, $ T $ has a unique fixed point $ z=1 $.  
\end{example}
\bibliographystyle{amsplain}

\end{document}